\documentclass[12pt,a4paper]{article}

\usepackage{tikz}
\usepackage{verbatim}
\usepackage{color}
\usepackage{amsmath}
\usepackage{amsthm}

\usepackage{amsfonts}
\usepackage{amssymb}
\usepackage{hyperref}
\usepackage{cleveref}
\usepackage{authblk}

\theoremstyle{definition}
\newtheorem{rem}{Remark}
\newtheorem{notation}{Notation}
\newtheorem{claim}{Claim}
\newtheorem{thm}{Theorem}
\newtheorem*{thm*}{Theorem}

\newtheorem*{question*}{Question}
\newtheorem*{observation*}{Observation}
\newtheorem{lem}[thm]{Lemma}
\newtheorem{cor}[thm]{Corollary}
\newtheorem{fact}[thm]{Fact}
\newtheorem{dfn}[thm]{Definition}

\usetikzlibrary{matrix}

\begin{document}
\title{The Additive  Structure of Integers with the Lower Wythoff Sequence}
\author[$\dagger$,1]{Mohsen Khani}
\author[$\dagger$,2]{Afshin Zarei}
\affil[$\dagger$]{Department of Mathematical Sciences, Isfahan University of Technology, Isfahan 84156-83111, Iran}
\affil[1]{mohsen.khani@iut.ac.ir}
\affil[2]{afshin.zarei@math.iut.ac.ir}
\date{\today}
\setcounter{Maxaffil}{0}
\renewcommand\Affilfont{\itshape\small}
\maketitle
\begin{abstract}
We have provided a model-theoretic proof for the decidability of the
additive structure of integers  together with the function $f$
mapping $x$ to $\lfloor \varphi x\rfloor$ where $\varphi$ is the golden ratio.
\end{abstract}
\section*{Introduction}
While the theory of the structure 
$(\mathbb{Z},+,\cdot)$ is famously undecidable, tame reducts of this structure have been subject of various literature,  see for example \cite{Co,H16,H15,KS}.
A classical result in this direction is the decidability of the
theory of the structure 
$(\mathbb{Z},+,\{p_n\}_{n\in \mathbb{N}},0,1)$,  known as the theory of $\mathbb{Z}$-groups.
In the mentioned structure, multiplication in $\mathbb{Z}$
is replaced by infinitely-many
unary predicates $p_n$, where
$p_n(x)$ holds if $x\stackrel{n}{\equiv} 0$.
More recent relevant results
are, for example, that there are no intermediate structures 
between the group of integers and Presburger arithmetic (Conant in \cite{Co}); and that
the  theory of integers with 
a predicate for 
prime numbers is decidable provided that Dickson's conjecture holds (Kaplan and Shelah in \cite{KS}). 
\par 
In this paper we prove
the decidability of the structure $(\mathbb{Z},+,f,0,1)$
where $f(x)=\lfloor \varphi x\rfloor$, and
$\varphi$ is the golden ratio.
We are already aware that this follows
from the  
decidability of the theory of $(\mathbb{R},\mathbb{Z},\alpha\mathbb{Z},+,<)$
for  a quadratic irrational number $\alpha$, as proved 
by Hieronymi in \cite{H16}.
His proof relies on the continued fractions and Ostrowski representations and 
 interpreting in the 
 structure $( \mathbb{N},\mathcal{P}(\mathbb{N}),\in ,s_{\mathbb{N}}) $. Since the latter  is decidable by a classical result of B\"uchi \cite{B62},
 so is the former. 
We have also been  later informed by
the referee that
the decidability of our  structure
can as well be obtained as
a consequence of two papers of Shallit et al., \cite{Shalitt1,Shalitt3} where
they even propose  an automata-based
decision algorithms for Fibonacci words.
\par 
Nevertheless, although we rely on facts on Fibonacci words
and the lower Wythoff sequence
to explain the properties of
the function $f$,
our approach is pure model-theoretic
and based on a quantifier-elimination result in a suitable language.
This approach was
suggested by Hieronymi, who guessed that a model-theoretic treatment of the properties of Beatty sequences would lead to the decidability of our structure.
\par 
Recall that any sequence of the form
$\mathcal{B}_r=(\lfloor rn\rfloor )_{n\in \mathbb{N}}$ 
for a positive irrational $r$, 
is called a \textit{Beatty sequence}. For $r>1$ and
$s={r}/({r-1})$,
 $(\mathcal{B}_r,\mathcal{B}_s)$  form a so-called
pair of complementary Beatty sequences;  that is
$\mathcal{B}_r\cup \mathcal{B}_s=\mathbb{N}$ and $\mathcal{B}_r\cap \mathcal{B}_s=\emptyset$
(see \cite{connell, connellii} for more on Beatty
sequences).
The Beatty sequence 
$\mathcal{B}_r$,
in the special case that $r=\varphi$,
is the golden ratio, is called the \textit{lower
Wythoff sequence}. 
\par 
We augment the language  of
$\mathbb{Z}$-groups
by
the unary function symbol $f$  and denote the obtained language by $\mathcal{L}$.
This choice of the language suggests that in order to have a chance for quantifier-elimination, 
we need to deal with  
systems of 
equations involving congruence relations and the function $f$.
It turns out that the solvability of such systems is closely related to
a classical theorem of
Kronecker 
(see \Cref{keronecker})
that
the set  of   decimal parts of elements of the form $\varphi n$, for $n\in \mathbb{N}$, is dense
in the unit interval $(0,1)$.  We will deploy this connection as a major means for our axiomatization.
\par 
The main idea we rely on is that the
order  of the decimal parts is definable in $\mathcal{L}$.
That is there is an $ \mathcal{L} $-formula $R(x,y)$ such that $R(m,n)$ holds for 
two integers $m,n$ if and only if the decimal part of
$ \varphi m $ is smaller than that  of $ \varphi n $.
 Hence we add a binary predicate 
$R(x,y)$ to $\mathcal{L}$ to obtain the language $\mathcal{L}^*$ (see \Cref{notation}), and our main theorem is the following.
\begin{thm*}
	The structure $(\mathbb{Z},+,f,R,\{p_n\}_{n\in \mathbb{N}},0,1)$ admits elimination of quantifiers.
\end{thm*}
\par 
The paper is structured as follows.
Basic facts
about the properties of $f$ 
are gathered in Section 1.
Note that as $f(-n)=-f(n)-1$ for each natural number $n$, 
in the lemmas on the properties of $f$ in sections 1,2, (until before \Cref{LemRxy}) we have restricted the domain to natural numbers.
In section 2, some auxiliary lemmas
are proved
to be used in Section 3 as the basis of our axiomatization.
The quantifier-elimination result
and the decidability that follows 
immediately from it are established in Section 4.
\section{Preliminaries on the Function $f$} 
By properties of the floor function, 
it is clear that 
for natural numbers  $m$ and $n$, we have
either $f(m+n)=f(m)+f(n)$, or 
 $f(m)+f(n)+1$. Hence, for each natural number $k$, there is 
 $0\leq \ell\leq k-1$ such that
 $f(kn)=kf(n)+\ell$. Of course $\ell$ is 
 the unique number such that $f(kn)\stackrel{k}{\equiv}\ell$.
\begin{lem}
\label{forf+x}
For every $m\in  \mathbb{N}$ there is $n\in \mathbb{N}$ such that either 
$m=f(n)$ or $m=f(n)+n$.
\end{lem}
\begin{proof}
As $\frac{\varphi}{\varphi-1}=\varphi+1$, $(\mathcal{B}_{\varphi},
\mathcal{B}_{\varphi+1})$ is a complementary pair of Beatty sequences,
which clearly means that each natural number $m$ is either equal to $f(n)$ or $f(n)+n$ for some natural number $n$.
\end{proof}
Note that \Cref{forf+x} holds also when one replaces 
$\mathbb{N}$ with $\mathbb{Z}$ and forces $m\neq -1$; simply because $f(-n)=-f(n)-1$ for all positive $n$.
\par 
Depending on whether $m$ belongs to the image of  $f$
or $f+\mathrm{id}$, where $\mathrm{id}$ denotes the identity function,  and
by the properties of  Beatty sequences one obtains a 
recursive definition for the function $f$ in natural numbers as in the following lemma.
\begin{lem}
\label{definitionoff}
$f(0)=0$, $f(1)=1$, and for each natural number $n>1$,
\mbox{$f(f(n)+n)=2f(n)+n$} and $f(f(n))=f(n)+n-1$.	 
\end{lem}
\begin{proof}
The fact that 
$f(f(n)+n))=2f(n)+n$ follows from 
\cite[Theorem 1]{connell}.
Now
$f(f(n)+n)$ equals either to $f(f(n))+f(n)$ or $f(f(n))+f(n)+1$.
The former cannot occur since
the images of $f$ and $f+\mathrm{id}$
are disjoint.
So $f(f(n)+n)=2f(n)+n=f(f(n))+f(n)+1$,
and the result follows.
\end{proof}
The lemma above implies    that for every $n \in \mathbb{N}$, $f(n)=\min \ \mathbb{N}\setminus \{f(i),f(i)+i: i<n\}$, hence in particular  $f$ is strictly increasing.
Again \Cref{definitionoff} also holds in $\mathbb{Z}$ but one needs to add that  $f(-1)=-2$.
\par 
At this point, we aim to establish 
the connection between 
the function $f$ and the Fibonacci sequence  $F_n$
with $F_0=F_1=1$.
This connection will play a major role in our proofs in this section of the properties of $f$
in natural numbers.
\par 
Consider the sequence
$(c_n)_{n\in \mathbb{N}}$, where $c_n=1$ if $n$ 
is in the image of $f$ and 0 otherwise.
By the properties of the floor function
and the fact that $\varphi>1$ it is easy to check that there are no successive zeros in $(c_n)_{n\in \mathbb{N}}$. 
A curious way to obtain $(c_n)_{n\in \mathbb{N}}$ is to
 start with the word 10 (of course of length 2) and then replace
1 with 10 and 0 with 1 (to obtain a word of length 3), and apply the same change to the word obtained (to obtain a word of length 5), and 
continue the same way.
So the length of each such word is a Fibonacci number and the last 
digit  alternates between 0 and 1.
So
 $c_{F_{2n+1}}=1$ and 
 $c_{F_{2n}}=0$ for each $n\in \mathbb{N}$, and
 $c_{F_{n}+i}=c_i$ for each $1\leq i\leq F_{n-1}$.  
 \par 
Meanwhile, note that each natural number $n$ has a unique \textit{Fibonacci representation};
that is, it can be uniquely written as a sum of non-successive decreasing Fibonacci numbers.
 To see this
one needs to find the largest Fibonacci 
number $F_{i_1}<n$, and write $n=F_{i_1}+G_1$.
Now let $F_{i_2}$ be the largest Fibonacci number less than $G_1$ and
write $n=F_{i_1}+F_{i_2}+G_2$ and continue with this procedure to end up with a 
Fibonacci number (see also \cite{Z}).
The explanation above on the sequence $(c_n)_{n\in \mathbb{N}}$
together with this  Fibonacci representation yields the following fact.
\begin{fact}
\label{last-index}
For each $n$, 
 the smallest index 
 appearing in  the unique {Fibonacci representation}
of $n$
determines $c_n$, where
$c_n=1$ if and only if
this index is odd.
\end{fact}
This in turn provides us with a concrete rule for $f$ as follows.
{
\begin{fact}\label{fact1}
\hfill
\begin{enumerate}
\item 
$f(F_i)=
F_{i+1}$ if  $i$ is even, and
$f(F_i)=F_{i+1}-1$, otherwise.
\item 
If $m$ has the Fibonacci representation  $m=F_{i_1}+F_{i_2}+\ldots +F_{i_\ell}$ with $i_1>i_2>\ldots >i_\ell$, then
$f(m)=
F_{i_1+1}+F_{i_2+1}+\ldots +F_{i_\ell+1}$ 
if $i_\ell$ \text{ is even}, and
$
f(m)=
F_{i_1+1}+F_{i_2+1}+\ldots +F_{i_\ell+1}-1$,
otherwise.
\end{enumerate}
\end{fact}
\begin{proof}
Note that $f(m)$ is equal to the index of the $m$th occurrence of $1$ in the sequence $(c_n)_{n\in \mathbb{N}}$. 
Now, by construction, the number of $1$'s  
in the sequence $(c_n)_{n\leq F_{i+1}}$
is $F_i$.
So, $f(F_i)=F_{i+1}$ if $i$ is even and 
$f(F_i)=F_{i+1}-1$ if $i$ is odd. Similarly the number of $1$'s
in $(c_n)_{n\leq N}$ with $N=F_{i_1+1}+F_{i_2+1}+\ldots +F_{i_{\ell}+1}$, is 
$F_{i_1}+F_{i_2}+\ldots +F_{i_{\ell}}$, and this implies the second item.
\end{proof}
It is worth reminding that our sequence $(c_n)_{n\in \mathbb{N}}$
is indeed the complement of the so-called \textit{Fibonacci word}, given by
$2+\lfloor{\varphi n}\rfloor-\lfloor
\varphi (n+1)\rfloor$.
\par 
We end this section by a simple, and
yet key fact (for a proof see  \cite{Wall}).
\begin{fact}
\label{fa2}
For any positive integer $k$ the sequence 
$(F_i\mod k )_{i\in \mathbb{N} }$ is periodic beginning with $0,1$.
\end{fact}}
\section{Auxiliary Lemmas for the Axiomatization}
The main lemmas in this section will
correspond to
the axioms we 
present
for our structure in the next section.
The first lemma below says that
the image of $f$ contains
elements in any congruence
class.
\begin{lem}\label{lem1}
For each $k,n\in \mathbb{N}$
there is $m\in  \mathbb{N}$ such that 
$f(m)\stackrel{k}\equiv n$.
\end{lem}
\begin{proof}
By \Cref{last-index},
if the smallest index in the Fibonacci representation of $n$ is odd,
 then $n$
 is in the image of $f$ and the equation is solved instantly; hence assume for the rest of the proof that this index is even.
 \par 
 Since  by \Cref{fa2} the Fibonacci sequence modulo $k$ is periodic,
  and the first two elements of this period are
 $0$ and $1$, there is a
  Fibonacci number $F_i$,
with $i$ larger than the largest index in 
the Fibonacci representation of $n$,  
    such that
 $F_i\stackrel{k}{\equiv} 1$. That is there is  $u\in \mathbb{N}$ such that 
  $F_i-1=ku$.
Because the index of the smallest Fibonacci number in the  representation of $n$ is even, i.e. $c_n=0$,
 we have
$c_{F_i-1+n}=c_{n-1}=1$. Hence $F_i-1+n$ is in the image of $f$ and $F_i-1+n \stackrel{k}{\equiv} n$. 
\end{proof}
The following lemma   generalizes the above 
in the way that
$m$ is also in a desirable congruence class.
\begin{lem}\label{lem5}
The  system 
\begin{align*}
	\begin{cases}
		x\stackrel{k}\equiv n\\
		f(x)\stackrel{k'}\equiv n'
	\end{cases}
\end{align*}
has a solution in
$\mathbb{N}$, 
for any $k,n,k',n'\in \mathbb{N}$. 
\end{lem}
\begin{proof}
It is obvious that $n$ is a solution of the equation $x\stackrel{k}{\equiv} n$. 
Let $F_{i_1}+F_{i_2}+\ldots +F_{i_\ell}$
be the Fibonacci representation of $n$ with $i_1>i_2>\ldots >i_\ell$. 
Suppose that
$n''$ is such that
 $f(n)\stackrel{k'}{\equiv}n''$. If $n''\neq n'$, by \Cref{fa2} we can find 
a Fibonacci number $F_j$ 
such that 
$F_j\stackrel{k}{\equiv}0$, $F_j\stackrel{k'}{\equiv}0$, $f(F_j)=F_{j+1}\stackrel{k'}{\equiv}1$, $j>i_1+1$, and 
$j$ is an even number. 
Now
since $F_j\stackrel{k}{\equiv} 0$, we have $n+F_j\stackrel{k}{\equiv}n$.
Also because $j>i_1 +1$ is even, 
by \Cref{fact1},
$f(n+F_j)=f(n)+f(F_j)\stackrel{k'}{\equiv}n''+1$.
This procedure gives, after finitely many steps,
a natural number $m$ such that 
$m\stackrel{k}{\equiv} n$ and $f(m)\stackrel{k'}{\equiv}n'$.
\end{proof}
Now, as it turns out, the lemma above has a close connection to 
the following fact about the 
distribution in the interval $(0,1)$ of 
the
decimal parts of 
$\varphi n$,  for all $n\in \mathbb{N}$ (as explained in 
\Cref{kron}).
Before mentioning the fact, let us fix some notation for the decimal parts.
\begin{notation}
We denote the decimal part of $\varphi n$ 
	by $[\varphi n]$; so $[\varphi n]=\varphi n-f(n)$.
\end{notation}

\begin{fact}(Kronecker \cite[Theorem 439]{hw})
\label{keronecker}
If 
$r$ is irrational, then $([r n])_{n\in \mathbb{N}}$ is dense in $(0,1)$.  
\end{fact}
To explain the connection  in question we need yet another lemma.
 \begin{lem}\label{lemker1}
 	For each $k$ and $0\leq i<k$,
 	$f(n)\stackrel{k}{\equiv}i$ if and only if 
 	$[\frac{\varphi}{k}n]\in (\frac{i}{k},\frac{i+1}{k})$.
 \end{lem}
 \begin{proof}
Write
 	$ \frac{\varphi}{k} n=u+r$, for some integer $ u $  and  $ 0<r<1 $. So 
 	$ f(n)\stackrel{k}{\equiv}i $ if and only if $ \varphi n=ku+kr $,
 	 with $ i<kr<i+1 $, that is $ \frac{i}{k}<r< \frac{i+1}{k}$. Therefore $ f(n)\stackrel{k}{\equiv}i $ if and only if $ \frac{\varphi}{k}n=u+r $ and $ \frac{i}{k}<r< \frac{i+1}{k} $. 
 \end{proof}
\begin{rem}\label{kron}
  By the lemma above, $[\varphi n]\in (\frac{i}{k},\frac{i+1}{k})$ if and only if $f(kn)\stackrel{k}{\equiv}i$.  Hence to find a natural number
  $n$ such that $[\varphi n]$ is
  in a desired small subinterval $(\frac{i}{k},\frac{i+1}{k})$ of $(0,1)$ one needs 
  to solve the following system of congruence-relation 
  equations: 
  \begin{align*}
  	\begin{cases}
   x\stackrel{k}\equiv 0 \\
f(x)\stackrel{k} \equiv i,		
  	\end{cases}
  \end{align*}
  and
  if $m$ is the solution of the system above then
  $n=\frac{m}{k}$ has the desired property.
This means that \Cref{lem5}  indeed implies Kronecker's theorem that $ \{[\varphi n]\}_{n\in \mathbb{N}} $ is dense in $ (0,1) $. 
It is an easy verification that Kronecker's theorem  also implies \Cref{lem5}.
\end{rem}
The remark above is interesting, because it suggests that
although
in our language
$\mathcal{L}$ 
it is not possible to have any symbol to refer to the decimal part of $\varphi x$,
 we are capable of finding an equivalent way of expressing in which interval $[\varphi x]$ lies.
Indeed many expressions about the function $f$ has an equivalent in terms of the decimal parts. For example,
$f(n+m)=f(n)+f(m)$ means that $[\varphi n]+[\varphi m]<1$, and similarly
$f(n+m)=f(n)+f(m)+1$ means that
$[\varphi n]+[\varphi m]>1$.  
As we see in the following lemma, much more can be said about the decimal parts already in the language $\mathcal{L}$.
\begin{lem}
	\label{LemRxy}
There is an $\mathcal{L}$-formula $R(x,y)$
such that for all $m,n\in \mathbb{Z}$,
$(\mathbb{Z},+,f,0,1)\models R(m,n)$
if and only if $[\varphi m]<[\varphi n]$.
\end{lem}
\begin{proof}
Simply let $R(x,y)$ be the following formula
\begin{align}
	\label{Rxy}
\forall z\Big(f(x+z)=f(x)+f(z)+1\to 
f(y+z)=f(y)+f(z)+1\Big).
\end{align}
We first show that if $[\varphi m]<[\varphi n]$, then $(\mathbb{Z},+,f,0,1)\models R(m,n)$. 
Note that when $[\varphi m]<[\varphi n]$
then for each $r\in \mathbb{Z}$, 
$[\varphi m]+[\varphi r]<[\varphi n]+[\varphi r]$.
Hence if $[\varphi m]+[\varphi r]>1$ then
$[\varphi n]+[\varphi r]>1$. But this is exactly what the formula $R(m,n)$ says.
\par 
For the other direction, note that if $[\varphi m]>[\varphi n]$, then $1- [\varphi m]<1-[\varphi n]$. Hence by Kronecker's theorem (\Cref{keronecker}), there is a
natural number $r$ such that $1- [\varphi m]<[\varphi r]<1-[\varphi n]$. But this means that $f(m+r)=f(m)+f(r)+1$ and $f(n+r)=f(n)+f(r) $, which means that the negation of $R(m,n)$ holds in $(\mathbb{Z},+,f,0,1)$.
\end{proof}
Thus the order of the decimal parts is definable in 
$(\mathbb{Z},+,f,0,1)$. 
We expand our language by the binary predicate $R(x,y)$ 
which is interpreted by formula \eqref{Rxy} above, and defines this order.
\begin{notation}\label{notation}
Let
	 $\mathcal{L}^*$ be the language $\mathcal{L}\cup \{R\}$.
 \end{notation} 
\begin{rem}\label{decorder}
	We interpret $R$ in $\mathbb{Z}$
as in \Cref{LemRxy}. 
We will later add to our axioms that $R$ is indeed a ``linear order relation''.
\end{rem}
Enriching the language with the predicate $R$ helps find equivalent expressions  in the language $\mathcal{L}^*$ 
to the following phrases (where $m_i,n_i\in \mathbb{N} $ and $\ell \in \mathbb{Z}$):
\begin{align}
& [\varphi x]+[\varphi y]<[\varphi z]\label{sum-of-decimals}\\
& m_1[\varphi x_1]+\ldots +m_k[\varphi x_k]<n_1[\varphi y_1]+\ldots +n_k[\varphi y_k]+\ell. \label{sum-decimals-tuple}
\end{align}
Phrase \ref{sum-of-decimals} can be expressed simply as follows:
\begin{align}
\Big( f(x+y)=f(x)+f(y)\Big) \wedge R(x+y,z)
\end{align}
Finding an equivalent to phrase \ref{sum-decimals-tuple} is also as easy; one needs to consider
cases $f(m_1x_1+\ldots+m_kx_k)=m_1f(x_1)+\ldots+m_kf(x_k)+\ell_1$, and
$f(n_1y_1+\ldots+n_ky_k)=n_1f(y_1)+\ldots+n_kf(y_k)+\ell_2$ for suitable $\ell_1,\ell_2$
and use the relation $R(m_1x_1+\ldots+m_kx_k,n_1y_1+\ldots+n_ky_k)$ accordingly.
\par 
Some more power of expression is provided using the following lemma.
\begin{lem}
\label{decimal-of-phi-f-x}
The decimal part of $\varphi f(n)$, for an integer $n$, is determined by the
 decimal part of $\varphi n$ as in 
the following:
\begin{align}\label{decf}
[\varphi f(n)]=
(1-\varphi)[\varphi n]+1
\end{align}
\end{lem}
\begin{proof} 
Note that $\varphi f(n)=\varphi (\varphi n-[\varphi n])=\varphi ^2n-\varphi[\varphi n]=\varphi n+n-\varphi[\varphi n]$, where the latter is the case because
$\varphi^2=\varphi+1$.
Hence, as $n$ is an integer, $[\varphi f(n)]=[\varphi n-\varphi [\varphi n]]$. 
\par 
If $[\varphi n]<\frac{1}{\varphi}$, then
obviously $[\varphi [\varphi n]]=\varphi [\varphi n]$.
Also it is clear that $[\varphi n]<\varphi [\varphi n]$, hence 
$[\varphi f(n)]=[\varphi n]-\varphi[\varphi n]+1$.
If $[\varphi n]>\frac{1}{\varphi }$, then $[\varphi [\varphi n]]=\varphi [\varphi n]-1<[\varphi n]$,
where the latter inequality is always the case.
So again
$[\varphi f(n)]=[\varphi n]-\varphi[\varphi n]+1$.
\end{proof}
Thus also the following phrase has an equivalent in the language $\mathcal{L}^*$ (where $m_i,n_i\in \mathbb{N}$ and $\ell,s\in \mathbb{Z}$):
\begin{align}
& m_1\varphi [\varphi x_1]+m_2[\varphi x_2]+\ldots+m_k[\varphi x_k]<\nonumber\\&n_1\varphi [\varphi y_1]+n_2[\varphi y_2]+\ldots+n_k[\varphi y_k]+\ell +\varphi s. \label{sum-decimals-tuple-with-phi} 
\end{align}
Note that the difference between the above phrase and
phrase \ref{sum-decimals-tuple} is that $\varphi$ itself appears as coefficient in two places, and a constant $\varphi s$
is added to the end.  To express
the above phrase one needs to simply replace $\varphi[\varphi x]$ with $[\varphi x]-[\varphi f(x)]+1$,
$\varphi[\varphi y]$ with $[\varphi y]-[\varphi f(y)]+1$ and $\varphi  s$ with $f(s)+[\varphi s]$,
to obtain 
 a similar phrase to \ref{sum-decimals-tuple}.

In  \Cref{T5} we will prove  that  the solvability
in $\mathbb{Z}$
 of a system of equations 
involving symbols of $\mathcal{L}^*$
 is expressible by a quantifier-free $\mathcal{L}^*$-formula. To explain the required argument more easily, we first deal with a simpler yet essential case in the following lemma.
\begin{lem}\label{rfsx}
 There is a quantifier-free formula $\Phi(y_1,y_2)$ in the language
$\mathcal{L}^*$
such that  for all integers $n_1,n_2$, $(\mathbb{Z},+,f,R,0,1)\models \Phi(n_1,n_2)$
if and only if
 the following system of equations 
has a solution in $\mathbb{Z}$,
\begin{align}\label{main-eq}
	\begin{cases}
 f(r_1x+s_1f(x)+n_1)=r_1f(x)+s_1f^2(x)+f(n_1)+j_1 \\
 f(r_2x+s_2f(x)+n_2)=r_2f(x)+s_2f^2(x)+f(n_2)+j_2
			\end{cases}
\end{align}
where  $r_1,s_1,j_1,r_2,s_2,j_2$ are fixed natural numbers.
\end{lem}
Note that the formula required in lemma above depends  on  
$r_1,s_1,j_1,r_2,s_2,j_2$, but for simplicity we have not reflected this dependence in the notation.
\begin{proof}
    The equations in \eqref{main-eq} can be rewritten  in terms of the decimal  parts as follows:
\begin{align*}
	\begin{cases}
j_1-[\varphi n_1]<r_1[\varphi x]+s_1[\varphi f(x)]<j_1+1-[\varphi n_1]\\
j_2-[\varphi n_2]<r_2[\varphi x]+s_2[\varphi f(x)]<j_2+1-[\varphi n_2].
\end{cases}
\end{align*}
Replacing $[\varphi f(x)]$ with 
$(1-\varphi)[\varphi x]+1$ as in
\Cref{decimal-of-phi-f-x},
we need to consider  the following system:
\begin{align}
\label{eqphix}
(r_1+s_1-s_1\varphi)[\varphi x]\in (j_1-[\varphi n_1] -s_1,j_1+1-[\varphi n_1] -s_1), \nonumber \\
(r_2+s_2-s_2\varphi)[\varphi x]\in  (j_2-[\varphi n_2] -s_2,j_2+1-[\varphi n_2] -s_2).
\end{align}
Thus the system is essentially of the form
\begin{align*}
& A_1[\varphi x]\in (B_1,C_1),\\
& A_2[\varphi x]\in (B_2,C_2)
\end{align*}
and  is solvable if either $\frac{B_1}{A_1}\in (\frac{B_2}{A_2},\frac{C_2}{A_2})$ or 
$\frac{C_1}{A_1}\in (\frac{B_2}{A_2},\frac{C_2}{A_2})$.
But then, each inequality needed to hold  (for example that $\frac{B_1}{A_1}> \frac{B_2}{A_2}$,
or equivalently $B_1A_2>B_2A_1$)
turns into an equality of the form \ref{sum-decimals-tuple-with-phi} and hence is expressible in the language 
$\mathcal{L}^*$.
\end{proof}
Now the same strategy (that is writing the equations in terms of the decimal parts and analyzing the obtained
 linear equations  in terms of disjoint or intersecting intervals) leads to the following corollary.
  Note that here 
other types of equations, for example of  the form $R(x,n)$ and $R(f(x),n)$, and  two congruence relation equations of the form
 $x\stackrel{r}{\equiv}j$ and $f(x)\stackrel{r'}{\equiv}j'$
are also added, but this does not essentially change the way we need to treat the system. Indeed all 
equations describe, in essence,  an inequality of the form
$ [\varphi x]\in (a,b) $ for suitable $a,b$.
\begin{cor}\label{T5}
There is a quantifier free  $\mathcal{L}^*$-formula 
$\theta(y_1,y_2,\ldots,y_{k})$
such that for all
$n_1,\ldots,n_{k}\in \mathbb{Z}$
we have $(\mathbb{Z},+,f,R,0,1)\models 
\theta(n_1,\ldots,n_{k})$ if and only if
the following system of equations 
 has a solution in $\mathbb{Z}$:
\begin{align}\label{eqT5}
\begin{cases}
f(r_1x+s_1f(x)+n_1)=r_1f(x)+s_1f^2(x)+f(n_1)+j_1\\
\vdots\\
f(r_{k-4}x+s_{k-4}f(x)+n_{k-4})=r_{k-4}f(x)+s_{k-4}f^2(x)+f(n_{k-4})+j_{k-4}\\
R(x,n_{k-3})\\
R(n_{k-2},x)\\
R(f(x),n_{k-1})\\
R(n_{k},f(x))\\
f(x)\stackrel{r_{k-3}}{\equiv}j_{k-3}\\
x\stackrel{r_{k-2}}{\equiv}j_{k-2}\\
\end{cases}
\end{align}
where $r_i,s_i, j_i\in \mathbb{N}$.
\end{cor}
\begin{proof}
Note that the last equation may be discarded, as one can replace $x$ with $r_{k-2}x'+j_{k-2}$ in the rest of 
the equations and change the system accordingly. 
The same is true for the last-but-one equation, as it  has an equivalent of the form $[\varphi x]\in (a,b)$ for some
rational $a,b$
and this information is
already in the rest of the equations.
So system \ref{eqT5} can be written in the following form:
\begin{align}
\begin{cases}
(r_1+s_1-s_1\varphi)[\varphi x]\in (j_1-[\varphi n_1] -s_1,j_1+1-[\varphi n_1] -s_1)\\
\vdots\\
(r_{k-4}+s_{k-4}-s_{k-4}\varphi)[\varphi x]\in (j_{k-4}-[\varphi n_{k-4}] -s_{k-4},j_{k-4}+1-[\varphi n_{k-4}] -s_{k-4})\\
[\varphi x]\in ([\varphi n_{k-2}],[\varphi n_{k-3}])\\
(1-\varphi) [\varphi x]\in([\varphi n_k],[\varphi n_{k-1}]).\\
\end{cases}
\end{align}
To solve this system one needs to check whether or not the corresponding intervals 
for $[\varphi x]$
have intersection. Now as in 
\Cref{rfsx} this can be described by an $\mathcal{L}^*$-formula.
\end{proof}
\section{Axiomatization}
We can now present an axiomatization 
$\mathbf{T}$
for our structure,
in the language $\mathcal{L}^*$
as in  \Cref{notation}.
The axioms are based on what we
developed in the previous two sections.
More specifically, the axiom-scheme 
$ (T1) $ below expresses the  basic properties of $(\mathbb{Z},+,\{p_n\}_{n\in \mathbb{N}},0,1)$
as a $\mathbb{Z}$-group.
 $(T2)$ and $(T3)$ express
 the main properties of the function $f$ based on \Cref{forf+x} and \Cref{definitionoff}.
$(T4)$
asserts that $R(x,y)$ is a  linear order relation, which, based on \Cref{decorder}, can be naturally thought of as the ``order of the decimal parts''.
 $ (T5) $ expresses that this order is dense.
 In other words,   $ (T5) $ expresses
  the Kronecker's theorem on the distribution
 of the decimal parts, based on \Cref{kron};  that is it says that if $R(a,b)$ holds (which can be thought of as $[\varphi a]<[\varphi b]$) then there is $c$ such that
 $R(a,c)$ and $R(c,b)$ hold (which again can be thought of as 
 $[\varphi a]<[\varphi c]<[\varphi b] $). Finally in the light of
\Cref{T5} the axiom-scheme $ (T6) $ expresses when a given system of equations has a solution.
\begin{dfn}\label{th2}
Let $\mathbf{T}$ be the theory 
obtained by the axioms expressing the following.
\begin{description}
\item[$(T1)$] The theory of $\mathbb{Z}$-groups,
 \item[$(T2)$]  $
\forall x \Big(x\neq -1\to \exists y \big( (x=f(y))\vee 
 (x=f(y)+y)\big)\Big)\wedge \forall x,y(f(x+y)=f(x)+f(y)\vee f(x+y)=f(x)+f(y)+1) $,
\item[$(T3)$]
$(f(0)=0)\ \wedge \ (f(1)=1)\ 
\wedge \ (f(-1)=-2)
\ \wedge\   \forall x \Big( f(f(x))=f(x)+x-1\  \wedge\  f(f(x)+x)=2f(x)+x\Big)$,
\item[$(T4)$] 
\begin{itemize}
\item $\forall x \ \neg R(x,x)$,
\item $\forall x,y \ \big(R(x,y) \to \neg R(y,x)\big)$,
\item $\forall x,y,z\ \big( R(x,z)\wedge R(z,y)\to R(x,y)\big)$,
\item $\forall x,y \ \big(R(x,y)\vee R(y,x)\big)$,
\end{itemize}
\item[$(T5)$]
$\forall x,y \Big(
R(x,y)\to \exists z\ (R(x,z)\wedge R(z,y))
\Big)$,
\item[$(T6)$] $\forall y_1,\cdots,y_{k}
\Bigg(\exists x \Big( \bigwedge\limits_{i=1}^{k-4} \big(f(r_ix+s_if(x)+y_i)=r_if(x)+s_if^2(x)+f(y_i)+j_i\big) \wedge R(x,y_{k-3})
\wedge R(y_{k-2},x)
\wedge 
R(f(x),y_{k-1})\wedge R(y_k,f(x))\wedge p_{r_{k-3}}(f(x)-j_{k-3})\wedge p_{r_{k-2}}(x-j_{k-2}) \Big)\leftrightarrow\theta(y_1,\ldots ,y_{k})\Bigg)$.
\end{description}
Note that $ (T1)$ and $(T6)$ are actually   axiom schemes.
\end{dfn}
\begin{thm}\label{ob1}
 $(\mathbb{Z},+,f,R,\{p_n\}_{n\in \mathbb{N}},0,1)$ is  a model of $\mathbf{T}$.
\end{thm}
The proof of the theorem above is clear by
the way we have established the axioms (and the explanation before \Cref{th2}). 
In the next section we have proved that
$\mathbf{T}$ eliminates quantifiers 
and this leads to the fact that 
$\mathbf{T}$ is complete and decidable.
\section{Quantifier-Elimination and Decidability}
For the rest of the paper,
let $\mathcal{M}_1$ and 
$\mathcal{M}_2$ be models of $\mathbf{T}$ and $\mathcal{M}$ be  a 
common substructure. We assume that
$\mathcal{M}_2$
is $|M|$-saturated.
To prove quantifier-elimination, we will show that
any finite system of equations in the language
$\mathcal{L}^*$ with parameters in $ M $ which has a solution in $M_1$ 
is also solvable in $M_2$.
\begin{lem}\label{cl1}
There is
$\mathcal{M}'\models (T1), (T2)$ such that
 $\mathcal{M}\subseteq \mathcal{M}'$ and $\mathcal{M}'\subseteq \mathcal{M}_i$ for $i=1,2$.
\end{lem}
\begin{proof}
Put $\mathcal{M}'=\{\frac{x}{n}|x\in M,\mathcal{M}_1,\mathcal{M}_2\models  p_n(x)\}$.
Indeed $\mathcal{M}'$ is the algebraically-prime model
of $(T1)$
in the language $\mathcal{L}$
containing 
$\mathcal{M}$.
\par 
We claim that
$\mathcal{M}'$ is closed under the function $f$, and hence 
bears an $\mathcal{L}$-structure.
 Suppose that $t=\frac{a}{n}\in M'$. Then $a=nt$ and $a\in M$. 
By properties  of $f$, we have  $\mathcal{M}_1\models f(a)=nf(t)+\ell$, where $\ell$ is the 
remainder of the division of $f(a)$ by $n$. 
 So $\mathcal{M}_1\models p_n(f(a)-\ell)$, and
$f	(a)-\ell \in M$. Therefore $f(t)=\frac{f(a)-\ell}{n}\in M'$. 
\par 
To prove Axiom $(T2)$,
let $a\neq -1$ be an arbitrary element in $M'$.
Since $\mathcal{M}_1$ is a model of $\mathbf{T}$,
there is $b\in M_1$ such that  $\mathcal{M}_1\models a=f(b)\vee a=f(b)+b$. 
We will show that  $b\in M'$.
\par 
If $\mathcal{M}_1\models a=f(b)$, then
by $(T3)$,
$\mathcal{M}_1\models f(a) =a+b-1$ 
and hence
$\mathcal{M}_1\models b =f(a)-a+1$.
It is clear that 
 $b\in M'$. 
\par 
If $\mathcal{M}_1\models a=f(b)+b$ 
then by
$(T3)$,
 $\mathcal{M}_1\models f(a)=f(f(b)+b)=2f(b)+b=a+f(b)$. Therefore $\mathcal{M}_1\models f(b)=f(a)-a$.
On the other hand, $a\in M'$ and $f(a)\in M'$ so $f(b)=f(a)-a\in M'$. 
Now since $f(b) \in M'$,  by 
the above argument
 we have $b\in M'$. 
 \par The second part of axiom $ (T2) $ is clearly inherited from $\mathcal{M}_1,\mathcal{M}_2$.
\end{proof}
\begin{thm}\label{th1}
The theory $\mathbf{T}$ admits elimination of quantifiers.
\end{thm}
\begin{proof}
	According to \Cref{cl1},
we  add to the assumptions at the beginning of this section that
 $\mathcal{M}\models (T1), (T2), (T3), (T4)$. Note that
 $(T3)$ and $(T4)$ come for free because they are universal and hence  inherited from $\mathcal{M}_1$ and $\mathcal{M}_2$. Now
assume that  an element $a\in M_1$ satisfies finitely-many equations, each  of which 
of one of the
forms in \Cref{main-system} below, with parameters $c$, $d$, $e$, $e'$, $g$, and $g'$ in $M$ and coefficients
$m,n,m',n',r,s,t,u,j$
 in $\mathbb{N}$. Note that the negation of each of the following equations (except for the last one) has the same format as itself.
\begin{align}\label{main-system}
\begin{cases}
x\stackrel{n}\equiv m\\
f(x)\stackrel{n'}\equiv m' \\
f(rx+sf(x)+c)=rf(x)+sf^2(x)+f(c)+j\\
R(x,e)\\
R(e',x)\\
R(f(x),g)\\
R(g',f(x))\\
tf(x)=ux+d
\end{cases}
\end{align}
Also notice that we do not get more involved equations (say in terms of the powers of $f$) simply because
the powers of $f$   reduce to one  by Axiom $(T3)$.
Now we aim to find $b\in M_2$ satisfying the equations above.
\par
We first claim that if the system above actually contains an equation
of the last form ($ tf(x)=ux+d $), then $a$ is already in $M$, hence
 $b\in M_2$
can be taken to be $a$ itself.
 Indeed such an equation has an ``algebraic nature'' where the rest of the equations, which only concern with the decimal parts can be thought of being ``non-algebraic'' (see \cite{kvz}).
\begin{claim}\label{fx-mnx}
If for $d\in M$, $\mathcal{M}_1\models f(a)=\frac{u}{t}a+d$  then $a\in M$.
\end{claim}
\begin{proof}
Suppose that $\mathcal{M}_1\models f(a)=\frac{u}{t}a+d$, so $\mathcal{M}_1\models f(f(a))=f(\frac{u}{t}a+d)$.
By axiom $(T3)$,
$\mathcal{M}_1\models f(f(a))=f(a)+a-1$, hence
similar to the proof of the previous lemma,  
 \[
\mathcal{M}_1\models f(a)+a-1=\frac{uf(a)+\ell}{t}+f(d)+j\]
for some integer $\ell$ and natural number $j$. Replacing 
$f(a)$ in both sides of the above formula with $\frac{u}{t}a+d $ we get a linear equation in terms of $a$.
This forces that $a\in M$,
 because the linear equation gives
$a$ by the divisibility relation and
 $\mathcal{M}$ is a model of $(T1)$.
\end{proof}
By the above claim, if there is an equation of the form $tf(x)=ux+b$ in  system (\ref{main-system}), then the solution of this system is already in $M$. 
Hence, in the rest we drop the last  equation from the system.
\par 
Meanwhile by the Chinese remainder theorem (which itself is deduced from the theory of $\mathbb{Z}$-groups and hence holds in our theory $\mathbf{T}$)
one reduces the congruence equation relations
for $f(x)$ and $x$ 
 to a single one. Similarly by the properties of a linear order,  one can assume that there is only one 
 equation of each form $R(x,e)$, $R(e',x)$, $R(f(x),g)$ and $R(g',f(x))$ in the system.
Now 
if follows from axiom-scheme $(T6)$ that 
the system has a solution in $M_2$. This is because the quantifier-free formula in the mentioned axiom is satisfied in $\mathcal{M}$, and this is  because the system clearly possesses a solution, that is $a$, in $ \mathcal{M}_1$.
\end{proof}
\begin{cor}
The theory $\mathbf{T}$  is complete and hence equivalent to ${\rm Th}(\mathbb{Z},+,f,R,\{p_n\}_{n\in \mathbb{N}},0,1)$.
\end{cor}
\begin{proof}
Quantifier-elimination implies that $\mathbf{T}$ is model-complete. Also  
$(\mathbb{Z},+,f,R,\{p_n\}_{n\in \mathbb{N}},0,1)$
is a prime model of $\mathbf{T}$, and the claim follows.
\end{proof}
\begin{cor}
The theory $\mathbf{T}$ is decidable.
\end{cor}
\begin{proof}
This follows from the fact that $\mathbf{T}$ is complete and
recursively enumerable. Note that to write $(T6)$ recursively, one only needs an algorithm to list all 
systems of equations and express when they are solvable. Expressing when a system is solvable 
only involves mentioning (using the power of the language $\mathcal{L}
^*$ as in \Cref{rfsx})
possible ways certain intervals have intersection.
 It is important to note that the algorithm is not needed to ``solve'' the system, but only express when it is solvable using finitely many conditions.
\end{proof}
\section*{Remarks}
\begin{enumerate}
\item In our earlier versions we claimed that
the structure $(\mathbb{Z},+,<,f,\{p_n\}_{n\in \mathbb{N}},0,1)$ eliminates quantifiers, but with the help of the referee's comments we found out that the proof was flawed.
 Adjusting our proof in the presence of the order of
 integers
 is not as straightforward as we first thought, and we leave this as the following question (whose answer we believe is positive).
 \begin{question*}
 Does the structure $(\mathbb{Z},+,<,f,R,\{p_n\}_{n\in \mathbb{N}},0,1)$ eliminate quantifiers?  
 \end{question*}
\item  
We do not know if $R(x,y)$ 
is equivalent to a quantifier-free formula in $\mathcal{L}$, 
but adding the binary function 
subtraction to the language, 
our quantifier-elimination result can be enhanced as below: 
\begin{observation*}[by A. Valizadeh]
The relation $R(x,y)$ is definable by the formula
$f(y-x)=f(y)-f(x)$. So (the proofs in this paper lead to the fact that) the structure
$(\mathbb{Z},+,-,f,\{p_n\}_{n\in \mathbb{N}},0,1)$
admits elimination of quantifiers.
\end{observation*} 
\item The proofs provided in the last version of this paper, are inspired by \cite{kvz}, where a much more difficult situation is dealt with in a similar manner. We decided to adopt the same technology here, as it made the proofs neater compared to our original proof.
\item The formula $R(x,y)$ suggests that the structure
$(\mathbb{Z},+,f,0,1)$ has the so-called ``order property'', which determines its place in terms of the model-theoretic classification of  theories. It is reasonable to ask whether this structure is NIP too.
\item We think that $\varphi$ can be replaced by any algebraic number, and the proofs will be essentially similar. But this needs to be checked.  As mentioned, in \cite{kvz} a much more general case is treated.

\end{enumerate}

\section*{Acknowledgements}
We would like to thank Philip Hieronymi, for bringing this question (and some others) to our attention, while
we were focused on a different one.
The question was initially suggested by him for the second author to work on 
in an academic  visit to Illinois, which finally turned out impossible.
\par 
We would also like to sincerely thank the anonymous referee whose comments
helped us improve this manuscript substantially.  His/Her questions and comments
lead us to this last version which we regard as much more mature than the earlier ones.
\par 
 We would like to thank
Ali Valizadeh whose contribution in \cite{kvz} lead to enhancing proofs
in this version.


\begin{thebibliography}{100}
\bibitem{Shalitt3}
Ch. F. Du, H. Mousavi,  E. Rowland, L. Schaeffer, and J. Shallit, 
``Decision algorithms for Fibonacci-automatic words, III: Enumeration and abelian properties'', \textit{Internat. J. Found. Comput. Sci.}, 27 (2016), no. 8, 943--963.
\bibitem{Wall}
D. D. Wall, ``Fibonacci Series Modulo $m$'', {\it The American Mathematical Monthly}, 67(6) (1960), 525--532.
\bibitem{Z}
E. Zeckendorf, ``Repr\'{e}sentation des nombres naturels par une somme de nombres de Fibonacci ou de nombres de Lucas'',
{\it Bull. Soc. Roy. Sci. Li\`{e}ge}, 41 (1972), 179--182. 
\bibitem{Co}
G. Conant, ``There are no intermediate structures between the group of integers and Presburger arithmetic,
{\it J. Symb. Log.}, 83 (2018), no. 1, 187--207.
\bibitem{hw}
G. H. Hardy and E. M. Wright. \textit{An Introduction to the Theory of Numbers.} Oxford University
Press, sixth edition, 2008.  
\bibitem{Shalitt1}
H. Mousavi,  E. Rowland, L. Schaeffer, and J. Shallit, 
``Decision algorithms for Fibonacci-automatic words, I: Basic results'',
\textit{RAIRO Theor. Inform. Appl.}, 50 (2016), no. 1, 39--66.
\bibitem{connell}
I. G. Connell, ``Some properties of Beatty sequences. I.'', {\it Canadian Mathematical Bulletin}, 2 (1959): 190--197.
\bibitem{connellii}
I. G. Connell, ``Some properties of Beatty sequences. II.'', {\it Canadian Mathematical Bulletin}, 3 (1960): 17--22.
\bibitem{KS}
I. Kaplan and S. Shelah, ``Decidability and classification of the theory of integers with primes'', {\it J. Symb. Log.},
82 (2017), no. 3, 1041--1050.
\bibitem{B62}
J. R. B\"{u}chi, ``On a decision method in restricted second order arithmetic", {\it Logic Methodology and Philosophy of Science(Proc. 1960 Internat. Cogr.)}, 1--11, Stanford Univ. Press, Stanford, Calif., 1962.
\bibitem{kvz}
M. Khani, A. N. Valizadeh, and A. Zarei, ``The additive structure of integers with a floor function'', \href{https://arxiv.org/abs/2110.01673}{arXiv:2110.01673}.
\bibitem{H16}
P. Hieronymi,  ``Expansions of the ordered additive group of real numbers by two discrete subgroups",
{\it J. Symb. Log.},  81 (2016), no. 3, 1007--1027.
\bibitem{H15}
P. Hieronymi, ``When is scalar multiplication decidable?'', \textit{Ann. Pure Appl. Logic}, 170 (2019), no. 10, 1162--1175.
\end{thebibliography}
\end{document}